\documentclass[11pt,a4paper,oneside]{amsart}
\usepackage{amsmath,amssymb}
\usepackage{bigints}
\DeclareMathOperator{\capacity}{Cap}

\DeclareMathOperator{\dist}{dist}

\DeclareMathOperator{\Lip}{Lip}

\newtheorem{theorem}{Theorem}[section]
\newtheorem{lemma}[theorem]{Lemma}

\newtheorem{definition}[theorem]{Definition}

\numberwithin{equation}{section}
\def\Xint#1{\mathchoice 
 {\XXint\displaystyle\textstyle{#1}}%
{\XXint\textstyle\scriptstyle{#1}}%
{\XXint\scriptstyle\scriptscriptstyle{#1}}%
 {\XXint\scriptscriptstyle\scriptscriptstyle{#1}}%
 \!\int}
\def\XXint#1#2#3{{\setbox0=\hbox{$#1{#2#3}{\int}$}
 \vcenter{\hbox{$#2#3$}}\kern-.5\wd0}}

 \def\dashint{\Xint-}

\title[Capacity in Besov and Triebel-Lizorkin spaces]{Capacity in Besov and Triebel-Lizorkin spaces with generalized smoothness}
\author{Nijjwal Karak and Debarati Mondal}

\thanks{We would like to thank the reviewer for the valuable comments and suggestions. N.K. would like to thank DST-SERB (SRG/2021/000118) and D.M. would like to thank BITS Pilani (ID no. 2020PHXF0038H) for the financial support. }
\begin{document}
\begin{abstract}
We prove a lower bound estimate for capacities in Haj\l asz-Besov, Haj\l asz-Triebel-Lizorkin and Haj\l asz-Sobolev spaces with generalized smoothness defined on metric spaces in terms of Netrusov-Hausdorff content or Hausdorff content. These results are improvements of the results obtained in \cite{LYY21}. 
\end{abstract}
\maketitle
\section{Introduction}
Function spaces with generalized smoothness have received a lot of attention in recent years due to its applications in various fields. In \cite{FL06}, the authors have studied the generalized Besov spaces $B^{(\sigma,N)}_{p,q}(\mathbb{R}^n)$ and Triebel-Lizorkin spaces $F^{(\sigma,N)}_{p,q}(\mathbb{R}^n)$ for the full range of parameters, where $\sigma$ is a weight sequence of positive numbers and $N$ is a sequence of strictly increasing positive numbers connected a suitable partition of unity, and they are used to replace the classical smoothness $\{2^{js}\}.$ Generalized Besov and Triebel-Lizorkin spaces also have been studied where smoothness is given by a parameter function, see \cite{CF88, CM04, CM04a, CF06, CL06, CK09}. Recently Li, Yang and Yuan have introduced Haj\l asz-Sobolev spaces $M^{\phi, p}(X)$, Haj\l asz-Besov spaces $N^{\phi}_{p,q}(X)$ and Haj\l asz-Triebel-Lizorkin spaces $M^{\phi}_{p,q}(X)$ with generalized smoothness on metric spaces, \cite{LYY22}. These spaces have been defined using Haj\l asz gradients which were first introduced by Haj\l asz \cite{Haj96} and later defined for fractional scales independently by Hu \cite{Hu03} and Yang \cite{Yan03} and then generalized to sequence of gradients to define Besov and Triebel-Lizorkin spaces \cite{KYZ11}. Please see Section 2 for the definition of these spaces.\\

In this paper we study the relation between the capacity defined on Haj\l asz-Sobolev spaces $M^{\phi, p}(X)$, Haj\l asz-Besov spaces $N^{\phi}_{p,q}(X)$ or Haj\l asz-Triebel-Lizorkin spaces $M^{\phi}_{p,q}(X)$ with generalized smoothness and Netrusov-Hausdorff content (or Hausdorff content). Haj\l asz-Sobolev capacity was first studied in \cite{KM96} and later it was used in \cite{KL02} to prove the existence of Lebesgue points of functions in Haj\l asz-Sobolev space $M^{1,p}(X),$ $p\in(1,Q],$ outside a set of Haj\l asz-Sobolev capacity zero, where $X$ is a doubling metric measure space and $Q$ is the doubling dimension of $X.$ This result leads to a series of related work on other function spaces such as fractional Haj\l asz-Sobolev spaces \cite{Pro06}, Orlicz-Sobolev spaces \cite{Moc11} as well as Haj\l asz-Besov and Haj\l asz-Triebel-Lizorkin spaces \cite{HKT17}. Let us recall here the definition of the capacities on these spaces from \cite{LYY21}. 
\begin{definition}
Let $0<s<\infty,$ $0<p,q< \infty,$ $\phi\in\mathcal{A}$ and $\mathcal{F}\in\{M^{\phi,p}, M^{\phi}_{p,q}, N^{\phi}_{p,q}\}.$ The $\mathcal{F}$-capacity of a set $E\subset X$ is
$$\capacity_{\mathcal{F}}(E)=\inf\{\Vert u\Vert_{\mathcal{F}(X)}^p: u\in\mathcal{G_F}(E)\},$$
where
$$\mathcal{G_F}(E)=\{u\in\mathcal{F}(X): u\geq 1 \,\,\text{on a neighbourhood of}\,\, E\}.$$
\end{definition}
Lower bound and upper bound estimates for Haj\l asz-Besov capacity ($N^s_{p,q}$-capacity) were found in \cite{Nuu16} in terms of generalized Netrusov-Hausdorff content, whereas the estimates for Haj\l asz-Triebel-Lizorkin capacity ($M^s_{p,q}$-capacity) were found in \cite{Kar20} in terms of generalized Hausdorff measure. These results were also extended in \cite{LYY21} for spaces with generalized smoothness and recently some of them have been improved in \cite{KM}. In this paper, we improve the results of \cite{LYY21} on the relation between the capacities and the Netrusov-Hausdorff content (or Hausdorff content) by relaxing the assumptions on the metric space. Here is our first result:
\begin{theorem}\label{main theorem_Besov}
Let $\phi\in \mathcal{A}_0$, $0<p<\infty$, $0<q<\infty$ and $\omega$ be any given function of admissible growth such that, for any $L\in \mathbb{Z}_{+}$, 
\begin{equation*}
\sum_{k\geq L}\frac{1}{\omega(2^{-k})}<\infty.
\end{equation*}
Let $x_0\in X,$ $0<R<\infty.$ Then there are constants $C>0$ and $c>0$ such that, for any compact set $E\subset B(x_0,R),$
\begin{equation}
\mathcal{H}^{h_\omega, q/p}_{cR}(E)\leq C \capacity_{N^\phi_{p,q}}(E)
\end{equation}
where, for any $r\in (0,R],$ $h_\omega(r)=[\phi(r)\omega(r)]^p.$
\end{theorem}
Note that, this result is an improvement of \cite[Theorem 7]{LYY21} in the sense that here we do not require any additional assumption of the existence of some non-empty annulus on the underlying metric space which was required for Theorem 7 of \cite{LYY21}. Also for Triebel-Lizorkin space, the lower bound is obtained, without the additional assumption on the underlying metric space, in terms of Hausdorff measure and the result turns out to be true as well for Sobolev spaces with generalized smoothness $M^{\phi,p}=M^{\phi}_{p,\infty}.$ 
\begin{theorem}\label{main theorem_Triebel}
Let $\phi\in \mathcal{A}_0$, $0<p<\infty$, $0<q<\infty$ and $\omega$ be any given function of admissible growth such that, for any $L\in \mathbb{Z}_{+}$, 
\begin{equation*}
\sum_{k\geq L}\frac{1}{\omega(2^{-k})}<\infty.
\end{equation*}
Let $x_0\in X,$ $0<R<\infty.$ Then there exist constants $c$ and $C$ such that, for any compact set $E\subset B(x_0,R),$
 \begin{equation*}
 H^{h_{\omega}}_{cR}(E)\leq C \capacity_{M^{\phi}_{p,q}}(E)
 \end{equation*}
where, for any $r\in (0,R],$ $h_\omega(r)=[\phi(r)\omega(r)]^p.$
\end{theorem}
\begin{theorem}\label{main theorem_Sobolev}
Let $\phi\in \mathcal{A}_0$, $0<p<\infty$, and $\omega$ be any given function of admissible growth such that, for any $L\in \mathbb{Z}_{+}$, 
\begin{equation*}
\sum_{k\geq L}\frac{1}{\omega(2^{-k})}<\infty.
\end{equation*} 
Let $x_0\in X,$ $0<R<\infty.$ Then there exist constants $c$ and $C$ such that, for any compact set $E\subset B(x_0,R),$
 \begin{equation*}
 H^{h_{\omega}}_{cR}(E)\leq C \capacity_{M^{\phi,p}}(E)
 \end{equation*}
where, for any $r\in (0,R],$ $h_\omega(r)=[\phi(r)\omega(r)]^p.$
\end{theorem}
\section{Notations and Preliminaries}
We assume throughout the article that $X=(X,d,\mu)$ is a doubling metric measure space equipped with a metric $d$ and a doubling Borel regular measure $\mu.$ Note that a Borel regular measure $\mu$ on a metric space $(X,d)$ is said to be a \textit{doubling measure} if every ball in $X$ has positive and finite measure and there exist a constant $c_d\geq 1$ such that
$\mu(B(x,2r))\leq c_d\,\mu(B(x,r))$
holds for each $x\in X$ and $r>0.$ 
\subsection{Smoothness weight sequences and functions}
We first recall some tools used to describe the generalized smoothness. Smoothness in terms of sequences can be found in \cite{FL06, KL87} and in terms of sequences can be found in \cite{LYY21, LYY22}. 
\begin{definition}
Let $E\in\{\mathbb{Z},\mathbb{Z_+}\}.$ A sequence of positive numbers $\{\sigma_j\}_{j\in E}$ is said to be admissible if there exist two positive constants $d_0$ and $d_1$ such that for any $j\in E$, $d_0\sigma_j\leq \sigma_{j+1}\leq d_1\sigma_j.$
\end{definition}
\begin{definition}
A continuous function $\phi:[0,\infty)\rightarrow [0,\infty)$ is said to be of admissible growth if $\{\phi(2^j)\}_{j\in \mathbb{Z}}$ is an admissible sequence and $\phi(t)\sim\phi(2^k) $ for any $k\in \mathbb{Z}$ and $t\in[2^k,2^{k+1})$ with the positive equivalence constants independent of both $t$ and $k$.
\end{definition}
Here we recall some indices from \cite{LYY21}. For any given sequence $\sigma:=\{\sigma_k\}_{k\in\mathbb{Z}}$ of positive numbers and any given function $\phi:[0,\infty)\rightarrow [0,\infty)$, let\\
$$\alpha_\phi:=\max\{\alpha^-_\phi,\alpha^+_\phi\}:=\max\left\{\limsup_{k\rightarrow-\infty}\frac{\phi(2^k)}{\phi(2^{k+1})},\limsup_{k\rightarrow\infty}\frac{\phi(2^k)}{\phi(2^{k+1})}\right\}$$ \\
and\\
$$\beta_\phi:=\max\{\beta^-_\phi,\beta^+_\phi\}:=\max\left\{\limsup_{k\rightarrow-\infty}\frac{\phi(2^{k+1})}{\phi(2^k)},\limsup_{k\rightarrow\infty}\frac{\phi(2^{k+1})}{\phi(2^k)}\right\}.$$

Throughout this article, we denote by $\mathcal{A}$ the class of all continuous and almost increasing functions $\phi:[0,\infty)\rightarrow [0,\infty)$ such that  $\phi(0)=0, \phi(1)=1,$ and $\{\phi(2^k)\}_{k\in \mathbb{Z}}$ is admissible. We also denote by $\mathcal{A}_\infty$ the class of all functions $\phi\in \mathcal{A}$ so that $\frac{\phi(t)}{t}$ almost decreases on $[0,\infty).$ Recall that a function $f:[0,\infty)\rightarrow [0,\infty)$ is said to be almost increasing (resp., almost decreasing) if there exists a positive constant $C_I\in[1,\infty)$ such that  for any $t_1,t_2\in[0,\infty)$ with $t_1\leq t_2$(resp.,$t_1\geq t_2$), $f(t_1)\leq C_I f(t_2).$ For any $r\in (0,\infty)$, let $\mathcal{A}_r$ be the set of all functions $\phi\in \mathcal{A}_\infty$ so that $\phi $ is of admissible growth and there exist an integer $i$ and two positive constants $X_{i}$ and $Y_{i}$ depending on $i$ and $r,$ such that 
\begin{equation}\label{A_0}
\left\{\sum_{j\geq i}[\phi(2^j)]^{-r}\right\}^{\frac{1}{r}}\leq X_{i}\quad \text{and}\quad \left\{\sum_{j\geq i}2^{-jr}[\phi(2^{-j})]^{-r}\right\}^{\frac{1}{r}}\leq Y_{i}.
\end{equation}
Also, let $\mathcal{A}_0$  be the class of all functions $\phi$ satisfying that $\alpha_\phi\in (0,1)$, $\beta^-_\phi\in(0,2)$ and $\phi$ is of admissible growth. Note that, by using the elementary inequality
\begin{equation}\label{for p less than 1}
\sum_{i\in\mathbb{Z}}a_i\leq \Big(\sum_{i\in\mathbb{Z}}a_i^{\beta}\Big)^{\frac{1}{\beta}}
\end{equation}
which holds whenever $a_i\geq 0$ for all $i\in\mathbb{Z}$ and $0<\beta\leq 1,$ we have $\mathcal{A}_{r_1}\subset \mathcal{A}_{r_2}\subset \mathcal{A}_{\infty}$ for any $r_1, r_2\in (0,\infty)$ with $r_1\leq r_2.$ Also, note that if $\phi\in\mathcal{A}_0,$ then $\phi\in\mathcal{A}_r,$ for any $r\in(0,\infty].$

\subsection{Haj\l asz-Sobolev, Haj\l asz-Besov and Haj\l asz-Triebel-Lizorkin spaces with generalized smoothness}
Here we recall the definitions of $\phi$-Haj\l asz gradients and the spaces with generalized smoothness from \cite{LYY22}.
\begin{definition}
A non-negative measurable function $g$ is called a $\phi$-Haj\l asz gradient of $u$ if there exists a set $E\subset X$ with $\mu(E)=0$ such that, for any $x,y\in X\setminus E,$
\begin{equation*}
|u(x)-u(y)|\leq \phi(d(x,y))[g(x)+g(y)].
\end{equation*}
We denote the collection of all $\phi$-Haj\l asz gradients of $u$ by $\mathcal{D}^\phi(u).$
\end{definition}
\begin{definition}
	Let $\phi\in\mathcal{A} $ and $0<p\leq \infty $. The homogeneous $\phi$-Haj\l asz-Sobolev space $\dot{M}^{\phi,p}(S)$ consists of all measurable, almost everywhere finite function $u:S\rightarrow \bar{\mathbb{R}}$ for which the semi norm 
	\begin{equation}
	\|u\|_{\dot{M}^{\phi,p}(S)}=\inf_{g\in \mathcal{D}^\phi(u)}\|g\|_{L^p(S)} 
	\end{equation}
	is finite.
	The non-homogeneous $\phi$-Haj\l asz-Sobolev space $M^{\phi,p}(S)$ is $\dot{M}^{\phi,p}(S)\cap L^p(S)$ equipped with the norm 
	\begin{equation}
	\|u\|_{M^{\phi,p}(S)}=\|u\|_{L^p(S)}+\|u\|_{\dot{M}^{\phi,p}(S)}.
	\end{equation}
	\end{definition}
	
Let $S\subset X$ be a measurable set. For $0<p,q\leq \infty$ and a sequence $\vec{f}=(f_k)_{k\in\mathbb{Z}}$ of measurable functions, we define
\begin{equation*}
\Vert (f_k)_{k\in\mathbb{Z}}\Vert_{L^p(S,l^q)}=\big\Vert \Vert (f_k)_{k\in\mathbb{Z}}\Vert_{l^q} \big\Vert_{L^p(S)}
\end{equation*}
and
\begin{equation*}
\Vert (f_k)_{k\in\mathbb{Z}}\Vert_{l^q(L^p(S))}=\big\Vert( \Vert f_k\Vert_{L^p(S)})_{k\in\mathbb{Z}}\big\Vert_{l^q}
\end{equation*}
where
\begin{equation*}
\Vert (f_k)_{k\in\mathbb{Z}}\Vert_{l^q}=
\begin{cases}
(\sum_{k\in\mathbb{Z}}\vert f_k\vert^q)^{1/q},& ~\text{when}~0<q<\infty,\\
\sup_{k\in\mathbb{Z}}\vert f_k\vert,& ~\text{when}~q=\infty.
\end{cases}
\end{equation*}	
\begin{definition}
A sequence of non negative measurable functions, $\vec{g}:=\{g_k\}_{k\in\mathbb{Z}}$ is called a $\phi$-Haj\l asz gradient sequence of $u$ if, for any $k\in\mathbb{Z},$ there exists a set $E_k\subset X$ with $\mu(E_k)=0$ such that , for any $x,y\in X\setminus E_k,$ with $2^{-k-1}\leq d(x,y)<2^{-k},$
\begin{equation*}
|u(x)-u(y)|\leq \phi(d(x,y))[g_k(x)+g_k(y)].
\end{equation*}
 We denote the collection of all $\phi$-Haj\l asz gradient sequences of $u$ by $\mathbb{D}^\phi(u).$
\end{definition}

\begin{definition}
	Let $\phi\in\mathcal{A} $ and $0<p,q\leq \infty $. The homogeneous $\phi$-Haj\l asz-Besov space $\dot{N}^\phi_{p,q}(S)$ consists of all measurable, almost everywhere finite function $u:S\rightarrow \bar{\mathbb{R}}$ for which the semi norm 
	\begin{equation}
	\|u\|_{\dot{N}^\phi_{p,q}(S)}=\inf_{(g_k)\in \mathbb{D}^\phi(u)}\|(g_k)\|_{l^q(L^p(S))} 
	\end{equation}is finite.
	The non-homogeneous $\phi$-Haj\l asz-Besov space $N^\phi_{p,q}(S)$ is $\dot{N}^\phi_{p,q}(S)\cap L^p(S)$ equipped with the norm 
	\begin{equation}
	\|u\|_{N^\phi_{p,q}(S)}=\|u\|_{L^p(S)}+\|u\|_{\dot{N}^\phi_{p,q}(S)}.
	\end{equation}
	\end{definition}
	\begin{definition}
	Let $\phi\in\mathcal{A} $ and $0<p,q\leq \infty $. The homogeneous $\phi$-Haj\l asz-Triebel-Lizorkin space $\dot{M}^\phi_{p,q}(S)$ consists of all measurable, almost everywhere finite function $u:S\rightarrow \bar{\mathbb{R}}$ for which the semi norm 
	\begin{equation}
	\|u\|_{\dot{M}^\phi_{p,q}(S)}=\inf_{(g_k)\in \mathbb{D}^\phi(u)}\|(g_k)\|_{L^p(S, l^q)} 
	\end{equation}is finite.
	The non-homogeneous $\phi$-Haj\l asz-Triebel-Lizorkin space $M^\phi_{p,q}(S)$ is $\dot{M}^\phi_{p,q}(S)\cap L^p(S)$ equipped with the norm 
	\begin{equation}
	\|u\|_{M^\phi_{p,q}(S)}=\|u\|_{L^p(S)}+\|u\|_{\dot{M}^\phi_{p,q}(S)}.
	\end{equation}
	\end{definition}
	
\subsection{$\gamma$-median}
\indent As our functions may fail to be locally integrable, we will use the concept of $\gamma$-median instead of integral averages. We recall here the definition and some of its properties which from \cite{HKT17, PT12, Nuu16}.
\begin{definition}
Let $0<\gamma\leq 1/2.$ The $\gamma$-median $m_u^{\gamma}(A)$ of a measurable, almost everywhere finite function $u$ over a set $A\subset X$ of finite measure is
 \begin{equation*}
m_u^{\gamma}(A)=\sup\{M\in\mathbb{R}:\mu(\{x\in A:u(x)<M\})\leq\gamma\mu(A)\}.
 \end{equation*}
 \end{definition}	

\begin{lemma}\label{gamma median 1}
Let $0<\gamma\leq 1/2$ and $A\subset X.$ Then for two measurable functions $u,v:A\rightarrow\mathbb{R},$ the $\gamma$-median has the following properties:\\
$(a)$ If $\gamma\leq\gamma ',$ then $m_u^{\gamma '}(A)\leq m_u^{\gamma}(A).$\\
$(b)$ If $u\leq v$ almost everywhere, then $m_u^{\gamma}(A)\leq m_v^{\gamma}(A).$\\
$(c)$ If $A\subset B$ and $\mu(B)\leq C\mu(A),$ then $m_u^{\gamma}(A)\leq m_u^{\gamma/C}(B).$\\
$(d)$ If $c\in\mathbb{R},$ then $m_u^{\gamma}(A)+c=m_{u+c}^{\gamma}(A).$\\
$(e)$If $c\in\mathbb{R},$ then $m_{cu}^{\gamma}(A)=cm_{u}^{\gamma}(A).$\\
$(f)$ $\vert m_u^{\gamma}(A)\vert\leq m_{\vert u\vert}^{\gamma}(A).$\\
$(g)$ If $u\in L^p(A),$ $p>0,$ then
$$m_{\vert u\vert}^{\gamma}(A)\leq\left(\gamma ^{-1}\dashint_A\vert u\vert ^p\,d\mu\right)^{1/p}.$$
$(h)$ If $u$ is continuous, then for every $x\in X,$
$$\lim_{r\rightarrow 0}m_u^{\gamma}(B(x,r))=u(x).$$
\end{lemma}	
We also recall the notion of discrete $\gamma$-median convolution introduced in \cite{HKT17}.  
\begin{definition}[\cite{LYY21}]
Let $\gamma\in(0,1/2]$ and $u\in L^0(X).$ The discrete $\gamma $- median convolution $u^\gamma_r$ of $u$ at scale $r\in(0,\infty)$ is defined by setting, for any $x\in X$
\begin{equation*}
u_r^\gamma(x):=\sum_{j\in\mathcal{J}}m^\gamma_u(B_j)\varphi_j(x),
\end{equation*}
where $\mathcal{J}$ is an index set, $\{B_j\}_{j\in \mathcal{J}}$ is a cover of $X$ with the radius $r$ such that $\sum_{j\in \mathcal{J}}\chi_{2B_j}$ is finite and  $\{\varphi_j\}_{j\in \mathcal{J}}$ is a partition of unity with respect to $\{B_j\}_{j\in \mathcal{J}}.$
\end{definition}

\subsection{Hausdorff content}
Netrusov-Hausdorff content was first introduced by Netrusov \cite{Net92, Net96} and later it was generalized using an increasing function $h$ in \cite{Nuu16}. For the generalized Hausdorff content (and measure), we refer to \cite{Rog98}.
\begin{definition}
Let $0<\theta<\infty,$ $0<R<\infty$ and $h: (0,\infty)\rightarrow (0,\infty)$ be an increasing function. The generalized Netrusov-Hausdorff content of a set $E\subset X$ is
$$\mathcal{H}^{h,\theta}_{R}(E)=\inf\Big[\sum_{i:2^{-i}<R}\Big(\sum_{j\in I_i}\frac{\mu(B(x_j,r_j))}{h(r_j)}\Big)^{\theta}\Big]^{\frac{1}{\theta}},$$
where $I_i=\{j\in\mathbb{N}:2^{-i}\leq r_j<2^{-i+1}\}$ and the infimum is taken over all coverings $\{B(x_j,r_j)\}$ of $E$ with $0<r_j\leq R.$ If $R=\infty,$ then we take the infimum over all coverings of $E$ and the first sum is over all $i\in\mathbb{Z}.$
\end{definition}
\begin{definition}
Let $0<R\leq\infty$ and $h: (0,\infty)\rightarrow (0,\infty)$ be an increasing function. The generalized Hausdorff content of a set $E\subset X$ is
$$H^{h}_R(E)=\inf\Big\{\sum_{j\in\mathbb{N}}\frac{\mu(B(x_j,r_j))}{h(r_j)}: E\subset\bigcup B(x_j,r_j), r_j\leq R\Big\}$$
and the generalized Hausdorff measure is $H^{h}(E)=\limsup_{R\rightarrow 0}H^{h}_R(E).$
\end{definition}
\noindent Note that $\mathcal{H}^{h,1}_R(E)=H^{h}_R.$ Also by using the elementary inequality \eqref{for p less than 1}
we have $\mathcal{H}^{h,\theta}_R(E)\leq H^{h}_R(E)$ if $\theta>1$ and $\mathcal{H}^{h,\theta}_R(E)\geq H^{h}_R(E)$ if $\theta<1.$


\section{Proofs}
We will first prove two lemmas (Lemma \ref{phi capacity lipschitz} and Lemma \ref{locally lipschitz}) to restrict the set of admissible functions for the capacities to locally Lipschitz functions. The proof of Lemma \ref{phi capacity lipschitz} follows by extending the proof of \cite[Theorem 4.8]{Nuu16} and using the convergence of approximations by discrete $\gamma$-median convolutions, \cite[Theorem 4]{LYY21}. We give the details here for the sake of completeness. 

\begin{lemma}\label{phi capacity lipschitz}
Let $\phi\in \mathcal{A}_0$, $0<p,q<\infty,\quad \mathcal{F}\in\{M^\phi_{p,q}(X), N^\phi_{p,q}(X)\}$ and $E\subseteq X$ be a compact set. Then
\begin{equation}
\capacity_\mathcal{F}(E)\approx \inf\{ \|u\|^p_\mathcal{F}:u\in \mathcal{\tilde{G}}_\mathcal{F}(E)\},
\end{equation}
where $\mathcal{\tilde{G}}_\mathcal{F}(E)=\{u\in \mathcal{G}_\mathcal{F}(E):u$ is locally Lipschitz$\}.$
\end{lemma}
\begin{proof} Since $\mathcal{\tilde{G}_F}(E)\subset \mathcal{G_F}(E)$, it is enough to prove  the $\lq\lq\gtrsim "$ part. Let $u\in \mathcal{G}_{\mathcal{F}}(E).$ Then there is an open set $U\supset E$ such that $u\geq 1$ in $U.$ Let $V=\{x:d(x,E)<d(E,X\setminus U)/2\}.$ If $x\in V$ and $r<d(E,X\setminus U)/8,$ then $B(y,2r)\subset U$ whenever $x\in B(y,2r).$ So, $u_r^\gamma\geq 1$ in $V$ when $r<d(E,X\setminus U)/8.$ So, $u_r^\gamma\in \mathcal{\tilde{G}}_{\mathcal{F}}(E)$ for small $r$ and hence by the convergence of approximations by discrete $\gamma$-median convolutions \cite[Theorem 4]{LYY21},
\begin{eqnarray*}
\inf\{\|v\|^p_\mathcal{F}:v\in \mathcal{\tilde{G}_F}(E)\}& \leq & \lim\inf_{i\rightarrow \infty}\|u^\gamma_{2^{-i}}\|^p_\mathcal{F}\\
&\leq & \lim\inf_{i\rightarrow \infty}C(\|u\|^p_\mathcal{F}+\|u^\gamma_{2^{-i}}-u\|^p_\mathcal{F}) 
 \\&\leq& C\|u\|^p_\mathcal{F}.
 \end{eqnarray*}
 The claim follows by taking the infimum over $u\in \mathcal{G}_\mathcal{F}(E).$
 \end{proof}

Note that, when $q=\infty,$ the approximation by the $\gamma$-median convolutions is not always true, even for the function $\phi(t)=t;$ see \cite[Example 3.5]{HKT17} for a counterexample. In this case, we use the denseness of the $\phi$-Lipschitz class of functions, proved in \cite[Theorem 2]{LYY21}. Recall that for given $\phi\in\mathcal{A},$ a function $u$ on $X$ is said to be in the $\phi$-Lipschitz class $\Lip_{\phi}(X)$ if there exists a constant $C>0$ such that for any two points $x$ and $y$ in $X,$ there holds
$$|u(x)-u(y)|\leq C\phi(d(x,y)).$$ 
\begin{theorem}\cite[Theorem 2]{LYY21}\label{dense}
Let $\phi$ be a modulus of continuity and $p\in(0,\infty).$ Then $\Lip_\phi(X)\cap M^{\phi,p}(X)$ is a dense subset of $M^{\phi,p}(X).$
\end{theorem}
We also need the following result to prove Lemma \ref{locally lipschitz}. 
\begin{lemma}\cite[Lemma 2(i)]{LYY21}\label{max,min}
Let $u,v \in L^0(X),\{g_k\}_{k\in\mathbb{Z}}\in \mathbb{D}^\phi(u)$ and $\{h_k\}_{k\in\mathbb{Z}}\in \mathbb{D}^\phi(v).$ Then $(\max\{g_k,h_k\})_{k\in\mathbb{Z}}\in \mathbb{D}^\phi(\min\{u,v\}).$
\end{lemma}
\begin{lemma}\label{locally lipschitz}
Let $\phi\in \mathcal{A},0<p<\infty$ and $E\subset X$ be a compact subset. Then 
\begin{equation*}
\capacity_{M^{\phi,p}}(E)\approx \inf\{ \|u\|^p_{M^{\phi,p}(X)}:u\in \mathcal{\tilde{G}'}_{M^{\phi,p}}(E)\},
\end{equation*}
where $\mathcal{\tilde{G}'}_{M^{\phi,p}}(E)=\{u\in M^{\phi,p}(E):u \,\,\text{is $\phi$-Lipschitz and}\,\, $u=1$ \,\, \text{on a neighbourhood of}\,\, E\}.$ 
\end{lemma}
\begin{proof}
It is easy to see that $\capacity_{M^{\phi,p}}(E)=\inf\{ \|u\|^p_{M^{\phi,p}(X)}:u\in \mathcal{G'}_{M^{\phi,p}(E)}\},$ where $\mathcal{G}'_{M^{\phi,p}}(E)=\{u\in M^{\phi,p}(X) :u= 1\,\,\text{in a neighbourhood of}\,\, E\}.$ Indeed, it is obvious that
$\capacity_{M^{\phi,p}}(E)\leq\inf\{\|w\|^p_{M^{\phi,p}(X)}:w\in \mathcal{G'}_{M^{\phi,p}}(E)\}$ and to prove the reverse inclusion, we see that for any $\epsilon>0,$ there exists  $u\in \mathcal{G}_{M^{\phi,p}}(E) $ such that $$\|u\|^p_{M^{\phi,p}(X)}\leq \capacity_{M^{\phi,p}}(E)+ \epsilon.$$
Hence $u'=\min\{1,u\}\in \mathcal{G'}_{M^{\phi,p}}(E)$ and by Lemma \ref{max,min},
\begin{equation*}
\inf\{\|w\|^p_{M^{\phi,p}(X)}:w\in \mathcal{G}'_{M^{\phi,p}}(E)\}\leq \|u'\|^p_{M^{\phi,p}(E)}\leq \|u\|^p_{M^{\phi,p}(E)}\leq \capacity_{M^{\phi,p}}(E)+ \epsilon,
\end{equation*}
which gives us the desired result by setting $\epsilon\rightarrow 0.$\\
Hence without loss of generality let $u\in \mathcal{G}'_{M^{\phi,p}}(E).$ So, we can take $u=1$ in a neighbourhood $U$ of $E.$ Then by Lemma \ref{dense}, there exists $u_j\in Lip_\phi(X)\cap M^{\phi,p}(X) $ such that $\|u_j-u\|_{M^{\phi,p}(X)}\rightarrow 0.$  Now choose a $\phi$-Lipschitz function $\eta,$ $0\leq\eta\leq 1$ such that $\eta=0$ in a neighbourhood $\tilde{U}\subset\subset U$ of $E$ and $\eta=1$ in $X\setminus U.$ Then the sequence $v_j=1-\eta(1-u_j)$
converges to $1-\eta(1-u)$ in $M^{\phi,p}(X),$ and $1-\eta(1-u)=u.$
Hence
 \begin{align*}
\inf_{v\in \mathcal{\tilde{G}}'_{M^{\phi,p}}(E)}\|v\|^p_{M^{\phi,p}(X)}\leq \liminf_{i\rightarrow\infty}\|v_j\|^p_{M^{\phi,p}(X)} & \leq \liminf_{i\rightarrow\infty}C(\|v_j-u\|^p_{M^{\phi,p}(X)} +\|u\|^p_{M^{\phi,p}(X)})\nonumber\\
& =C\|u\|^p_{M^{\phi,p}(X)}.
\end{align*} 
The lemma follows as the reverse inclusion is obvious.
\end{proof}

We also need the following Poincar\'e-type inequalities: the first one is proved in \cite[Lemma 15]{LYY21} and the second one can be proved by extending the proofs available for $\phi(t)=t,$ see \cite[Theorem 4.5]{Nuu16} and also \cite[Lemma 3.1]{KM}.

\begin{lemma}\label{phi integral average}
Let $\alpha_\phi\in (0,1)$ and $\gamma\in (0,1/2].$ \\
$(i)$ For any given $\lambda\in(0,\infty),$ there exists a positive constant $C$ such that, for any $k\in \mathbb{Z},u\in L^0(X),g\in \mathcal{D}^\phi(u)$ and $x\in X$
\begin{equation*}
\inf_{c\in \mathbb{R}}m^\gamma_{|u-c|}(B(x,2^{-k}))\leq C\phi(2^{-k})\left\{\dashint_{B(x,2^{-k+1})}[g(y)]^\lambda d\mu(\gamma)\right\}^{\frac{1}{\lambda}}.
\end{equation*}
$(ii)$ Let $u\in \{M^{\phi}_{p,q}(X), N^{\phi}_{p,q}(X)\}$ and $0<p,q<\infty.$ Then there exist a positive constant $C$ and a sequence $(g_k)_{k\in\mathbb{Z}}\in \mathbb{D}^\phi(u)$ such that for any $k\in\mathbb{Z},$
\begin{equation*}
\inf_{c\in \mathbb{R}}m^\gamma_{|u-c|}(B(x,2^{-k}))\leq C\phi(2^{-k})\left\{\dashint_{B(x,2^{-k+1})}[g_k(y)]^p d\mu(\gamma)\right\}^{\frac{1}{p}}.
\end{equation*}
\end{lemma}

{\bf{Proof of Theorem \ref{main theorem_Besov}:}} For simplicity, we assume that $R=2^{-m}$ for some $m\in \mathbb{Z}$. Without any loss of generality, we can assume that $x_0\notin E.$ By Lemma \ref{phi capacity lipschitz}, there exists a locally Lipschitz function $v\in N^\phi_{p,q}(X)$ with $v\geq 1$ on a neighbourhood of $E$ which satisfies
\begin{equation}\label{from lemma_Besov}
\|v\|^p_{N^\phi_{p,q}(X)}\leq C \capacity_{N^\phi_{p,q}}(E)+\delta \quad\text{for every}\quad \delta>0.
\end{equation} 
Now, we take a bounded $L$-Lipschitz function $\psi$ which satisfies 
\begin{equation}
\psi=0 \quad\text{on}\quad B(x_0,\frac{1}{4}\dist(x_0,E))\quad\text{and}\quad \psi=1 \quad\text{outside}\quad B(x_0,\frac{1}{2}\dist(x_0,E)). 
\end{equation}
Consider, $u=v\psi $.
It is clear that $u\geq 1$ on $E$ and we claim that there exists  a  $\phi$-Haj\l asz gradient sequence $(g_k)_{k\in \mathbb{Z}}$ of $u$ which takes $0$ in $B(x_0,\frac{1}{4}\dist(x_0,E))$ and satisfies
\begin{equation}\label{claim_Besov}
 \|(g_k)\|_{l^q(L^p(X))}\leq C_1 \|v\|_{N^{\phi}_{p,q}(X)}.
\end{equation}
Proof of the claim:\\
First we will prove that the sequences $(\rho_k)_{k\in \mathbb{Z}}$, $(\sigma_k)_{k\in \mathbb{Z}}$ defined by
\begin{equation}\label{rhok_Besov}
\rho_k(t)=
\begin{cases}
  \|\psi\|_{L^{\infty}(X)}h_k+ \frac{2^{-k}C_I}{\phi(2^{-k-1})}L|v| & \text{if $t\in X\setminus B(x_0,\frac{1}{4}\dist(x_0,E))$},\\
  0& \text{if $t\in B(x_0,\frac{1}{4}\dist(x_0,E))$}
 \end{cases}
\end{equation}
and
\begin{equation}\label{sigmak_Besov}
\sigma_k(t)=
\begin{cases}
  h_k+\frac{2C_I}{\phi(2^{-k-1})}\|\psi\|_{L^{\infty}(X)}|v| & \text{if $t\in X\setminus B(x_0,\frac{1}{4}\dist(x_0,E))$},\\
  0& \text{if $t\in B(x_0,\frac{1}{4}\dist(x_0,E))$}
 \end{cases}
\end{equation}
are $\phi$-Haj\l asz gradient sequences of $u$, where $(h_k)_{k\in \mathbb{Z}}$ is a $\phi$-Haj\l asz gradient sequence of $v$. Firstly, we will assume $y,z\in X\setminus B(x_0,\frac{1}{4}\dist(x_0,E))$ with $2^{-k-1}<d(y,z)\leq 2^{-k}$. The other cases follow similarly.\\
Now, since $\phi $ is almost increasing, 
\begin{eqnarray*}
|u(y)-u(z)|& = & |v(y)\psi(y)-v(z)\psi(z)|\\
&\leq &  |v(y)|\,|\psi(y)-\psi(z)|+|\psi(z)|\,|v(y)-v(z)|
 \\&\leq& |v(y)|\,|Ld(y,z)|+\|\psi\|_{L^{\infty}(X)}\phi(d(y,z))(h_k(y)+h_k(z))\\
&\leq & \phi(d(y,z))(|v(y)|L\frac{2^{-k}C_I}{\phi(2^{-k-1})}+\|\psi\|_{L^{\infty}(X)}(h_k(y)+h_k(z)))\\&\leq&
 \phi(d(y,z))(\rho_k(y)+\rho_k(z)).
 \end{eqnarray*}
 Also, 
 \begin{eqnarray*}
  |u(y)-u(z)|& = &|v(y)\psi(y)-v(z)\psi(z)|\\
 &\leq &|v(y)|\,|\psi(y)-\psi(z)|+|\psi(z)|\,|v(y)-v(z)|\\
 &\leq & 2\|\psi\|_{L^{\infty}(X)}|v(y)|+\phi(d(y,z))(h_k(y)+h_k(z))\\
&\leq & \phi(d(y,z))(2\|\psi\|_{L^{\infty}(X)}|v(y)|\frac{C_I}{\phi(2^{-k-1})}+h_k(y)+h_k(z))\\
 &\leq &\phi( d(y,z))(\sigma_k(y)+\sigma_k(z)).
 \end{eqnarray*}
 Therefore, the sequences $(\rho_k)_{k\in \mathbb{Z}}$ and $(\sigma_k)_{k\in \mathbb{Z}}$ are $\phi$-Haj\l asz gradient sequences of $u$. Now we will prove \eqref{claim_Besov} with the sequence $(g_k)_{k\in \mathbb{Z}}$, where we define, for any $i\in\mathbb{Z},$
\begin{equation}\label{gk}
g_k=\begin{cases}\rho_k &\text{if $ k\geq i$}\\\sigma_k &\text{if $k<i.$}\end{cases}
\end{equation}
In fact,
\begin{eqnarray*}
& & \|(g_k)_{k\in\mathbb{Z}}\|_{l^q(L^p(X))} \\
 &=& \left(\sum_{k\in\mathbb{Z}}\|g_k\|^q_{L^p(X)}\right)^{\frac{1}{q}}\nonumber\\
&\leq & c\left[\left(\sum_{k\geq i}\|h_k+\frac{2^{-k}}{\phi(2^{-k-1})}|v|\,\|^q_{L^p(X)}\right)^{\frac{1}{q}}+\left(\sum_{k<i}\|h_k+\frac{1}{\phi(2^{-k-1})}|v|\,\|^q_{L^p(X}\right)^{\frac{1}{q}}\right]\nonumber\\
&\leq & c\left[\left(\sum_{k\in\mathbb{Z}}\|h_k\|^q_{L^p(X)}\right)^{\frac{1}{q}}+\|v\|_{L^P(X)}\left(\sum_{k\geq i}\frac{2^{-kq}}{\phi(2^{-k-1})^q}\right)^{\frac{1}{q}}+\|v\|_{L^P(X)}\left(\sum_{k<i}\frac{1}{\phi(2^{-k-1})^q}\right)^{\frac{1}{q}}\right]\nonumber\\
&\leq & c\left[\|(h_k)_{k\in \mathbb{Z}}\|_{l^q(L^p(X))}+\|v\|_{L^p(X)}\right],
 \end{eqnarray*}
where in the last inequality we have used the inequalities in \eqref{A_0} after choosing $i$ such that $2^{i-1}\leq L<2^i$. Since $(h_k)$ is a $\phi$-Haj\l asz gradient sequence of $v$, it completes proof of \eqref{claim_Besov}.\\

Let $x\in E,$ $y\in B(x,2^{-m})$ and $z\in B(x_0,\frac{1}{4}\dist(x_0,E)).$ Hence $d(y,z)<2^{-m+2}.$ Therefore there exists $k_0\geq m-2$ such that $ 2^{-k_0-1}\leq d(y,z)<2^{-k_0}.$ So,
\begin{eqnarray*}
1\leq u(x)\leq |u(x)-m^{\gamma}_u(B(x,2^{-k_0}))|+|m^{\gamma}_u(B(x,2^{-k_0}))|.
 \end{eqnarray*}
Since $u$ is continuous, $x$ is a generalized Lebesgue point of $u$. By, properties of $\gamma$-median,
\begin{align*}
|u(x)-m^{\gamma}_u(B(x,2^{-k_0}))|&\leq \sum_{k\geq k_0}|m^{\gamma}_u(B(x,2^{-k-1}))-m^{\gamma}_u(B(x,2^{-k}))|\nonumber\\
&\leq\sum_{k\geq k_0}m^{\gamma}_{|u-m^{\gamma}_u(B(x,2^{-k}))|}(B(x,2^{-k-1}))\nonumber\\
&\leq\sum_{k\geq k_0}m^{\gamma/c_d}_{|u-m^{\gamma}_u(B(x,2^{-k}))|}(B(x,2^{-k}))\nonumber\\
& \leq 2\sum_{k\geq k_0}\inf_{c\in \mathbb{R}}\left[m^{\gamma'}_{|u-c|}(B(x,2^{-k}))+m^{\gamma'}_{|u-c|}(B(x,2^{-k}))\right]\nonumber\\&
\end{align*}
and then from Lemma \ref{phi integral average} we get
\begin{equation}\label{generalized Lebesgue pt_Besov}
|u(x)-m^{\gamma}_u(B(x,2^{-k_0}))|\leq c\sum_{k\geq k_0}\phi(2^{-k})\left(\dashint_{B(x,2^{-k+1})}g^p_k d\mu\right)^{\frac{1}{p}},
\end{equation}
where $\gamma'=\min\{\gamma,\gamma/c_d\}$.
Using the definition of $\phi$-Haj\l asz gradient,
\begin{eqnarray*}
|u(y)|=|u(y)-u(z)|\leq \phi(d(y,z)) g_{k_0}(y).
\end{eqnarray*}
So, for a.e. $y, z$ with $d(y,z)<2^{-k_0},$ $|u(y)|\leq c\phi(2^{-k_0})g_{k_0}(y)$, since $\phi$ is almost increasing. Hence, using the properties of $\gamma$-median we obtain
 \begin{align}\label{gamma median_Besov} |m^{\gamma}_u(B(x,2^{-k_0}))|
&\leq m^\gamma_{c\phi(2^{-k_0})g_{k_0}}(B(x,2^{-k_0})) \nonumber\\
&=c \phi(2^{-k_0}) m^\gamma_{g_{k_0}}(B(x,2^{-k_0}))\nonumber\\
&\leq c\phi(2^{-k_0})\left(\gamma^{-1}\dashint_{B(x,2^{-k_0})}g^p_{k_0}d\mu\right)^{\frac{1}{p}}\nonumber\\
&\leq c\sum_{k\geq k_0}\phi(2^{-k})\left(\gamma^{-1}\dashint_{B(x,2^{-k})}g^p_k d\mu\right)^{\frac{1}{p}}.
\end{align}
Summing \eqref{generalized Lebesgue pt_Besov} and \eqref{gamma median_Besov}
\begin{align}
1 \leq u(x) & \leq c\sum_{k\geq k_0}\phi(2^{-k})\left(\dashint_{B(x,2^{-k+1})}g^p_kd\mu\right)^{\frac{1}{p}} \nonumber\\
& \leq c\left(\sum_{k\geq k_0}h_\omega(2^{-k+1})^{-\frac{1}{p}}\phi(2^{-k})\right)\sup_{k\geq k_0}h_\omega(2^{-k+1})^{\frac{1}{p}}\left(\dashint_{B(x,2^{-k+1})}g^p_kd\mu\right)^{\frac{1}{p}}\nonumber\\
&\leq c\left(\sum_{k\geq k_0}\omega(2^{-k})^{-1}\right)\sup_{k\geq k_0}h_\omega(2^{-k+1})^{\frac{1}{p}}\left(\dashint_{B(x,2^{-k+1})}g^p_kd\mu\right)^{\frac{1}{p}}\nonumber\\
&\leq c \sup_{k\geq k_0}h_\omega(2^{-k+1})^{\frac{1}{p}}\left(\dashint_{B(x,2^{-k+1})}g^p_kd\mu\right)^{\frac{1}{p}}.
\end{align}
Now for every $x\in E$, there is a ball $B(x,2^{-k_x+1})$ with $k_x\geq k_0$ such that 
$$\frac{\mu(B(x,2^{-k_x+1})}{h_\omega(2^{-k_x+1})}\leq c \int_{B(x,2^{-k_x+1})}g^p_{k_x}d\mu .$$ 
By $5B$-covering lemma, there exists a countable family of disjoint balls $B_j=B(x_j,2^{-k_{x_j}+1})$ of radii $r_j=2^{-k_{x_j}+1}\leq 2^{-m+2}$, such that the dilated balls $5B_j$ cover the set $E$. Now we set $I_i=\{j:2^{-i}\leq 5r_j\leq 2^{-i+1}\}$. Then $k_{x_j}=i+3$ for $j\in I_j$. Since $\phi $ is increasing, $\omega$ is admissible, $\mu$ is doubling and the balls are disjoint, we obtain
\begin{align*}
\sum_{j\in I_i}\frac{\mu(5B_j)}{h_\omega(5r_j)}& \leq c\sum_{j\in I_i}\frac{\mu(B_j)}{h_\omega(r_j)}\nonumber\\
& \leq c\sum_{j\in I_i}\int_{B_j}g^p_{i+3}d\mu\nonumber\\
& \leq c \|g_{i+3}\|^p_{L^p(X)},
\end{align*}
where in the penultimate inequality we have used the disjointness of the balls $B_j.$ Summing over $i,$ we get
\begin{align*}
\sum_{2^{-i}<5.2^{-m+2}}\left(\sum_{j\in I_i}\frac{\mu(5B_j)}{h_\omega(5r_j)}\right)^{\frac{q}{p}}&\leq c\sum_{i\in\mathbb{Z}}\|g_{i+3}\|^q_{L^p(X)}
\end{align*}
and hence $\mathcal{H}^{h_\omega,q/p}_{20R}(E)\leq c \left(\sum_{i\in \mathbb{Z}}\|g_{i+3}\|^q_{L^p(X)}\right)^{p/q}.$ Now, taking $\delta\rightarrow 0$  in \eqref{from lemma_Besov} and using \eqref{claim_Besov} we get the desired result. \qed\\

{\bf{Proof of Theorem \ref{main theorem_Triebel}:}} For simplicity, we assume that $R=2^{-m}$ for some $m\in \mathbb{Z}$. Let $E\subset B(x_0,2^{-m})$ be a compact set. Without any loss of generality, we can assume that $x_0\notin E.$ By Lemma \ref{phi capacity lipschitz} there exists a locally Lipschitz function $v\in M^\phi_{p,q}(X)$ with $v\geq 1$ in a neighbourhood of $E$ which satisfies
\begin{equation}\label{from lemma_Triebel}
\|v\|^p_{M^\phi_{p,q}(X)}\leq C \capacity_{M^\phi_{p,q}}(E)+\delta \quad\text{for every}\quad \delta>0.
\end{equation} 
Now, we take a bounded $L$-Lipschitz function $\psi$ which satisfies 
\begin{equation}
\psi=0 \quad\text{on}\quad B(x_0,\frac{1}{4}\dist(x_0,E))\quad\text{and}\quad \psi=1 \quad\text{outside}\quad B(x_0,\frac{1}{2}\dist(x_0,E)). 
\end{equation}
Consider, $u=v\psi $.
It is clear that $u\geq 1$ on $E$ and we claim that there exists  a  $\phi$-Haj\l asz gradient sequence $(g_k)_{k\in \mathbb{Z}}$ of $u$ which takes $0$ in $B(x_0,\frac{1}{4}\dist(x_0,E))$ and satisfies
\begin{equation}\label{claim_Triebel}
 \|(g_k)\|_{L^p(X, l^q)}\leq C_1 \|v\|_{M^{\phi}_{p,q}(X)}.
\end{equation}
In fact, the sequence $\{g_k\}$ defined in \eqref{gk} is the $\phi$-Haj\l asz gradient satisfies \eqref{claim_Triebel}:
\begin{eqnarray*}
& & \left(\sum_{k\in\mathbb{Z}}g_k^q\right)^{1/q}\\
&\leq & c\left[\left(\sum_{k\geq i}(h_k+\frac{2^{-k}}{\phi(2^{-k-1})}|v|\,)^q\right)^{\frac{1}{q}}+\left(\sum_{k<i}(h_k+\frac{1}{\phi(2^{-k-1})}|v|\,)^q\right)^{\frac{1}{q}}\right]\nonumber\\
&\leq & c\left[\left(\sum_{k\in\mathbb{Z}}h_k^q\right)^{\frac{1}{q}}+|v|\left(\sum_{k\geq i}\frac{2^{-kq}}{\phi(2^{-k-1})^q}\right)^{\frac{1}{q}}+|v|\left(\sum_{k<i}\frac{1}{\phi(2^{-k-1})^q}\right)^{\frac{1}{q}}\right]
 \end{eqnarray*}
 and hence
\begin{equation*}
\|g_k\|_{L^p(X,l^q)}\leq c \left[\|(h_k)_{k\in\mathbb{Z}}\|_{L^p(X,l^q)} + \|v\|_{L^p(X)}\right],
\end{equation*}
where in the last inequality we have used the inequalities in \eqref{A_0} after choosing $i$ such that $2^{i-1}\leq L<2^i$.\\

Let $x\in E,$ $y\in B(x,2^{-m})$ and $z\in B(x_0,\frac{1}{4}\dist(x_0,E)).$ Hence $d(y,z)<2^{-m+2}.$ Therefore there exists $k_0\geq m-2$ such that $ 2^{-k_0-1}\leq d(y,z)<2^{-k_0}.$ So,
\begin{eqnarray*}
1\leq u(x)\leq |u(x)-m^{\gamma}_u(B(x,2^{-k_0}))|+|m^{\gamma}_u(B(x,2^{-k_0}))|.
 \end{eqnarray*}
Since $u$ is continuous, $x$ is a generalized Lebesgue point of $u$. By, properties of $\gamma$-median,
\begin{align*}
|u(x)-m^{\gamma}_u(B(x,2^{-k_0}))|&\leq \sum_{k\geq k_0}|m^{\gamma}_u(B(x,2^{-k-1}))-m^{\gamma}_u(B(x,2^{-k}))|\nonumber\\
&\leq\sum_{k\geq k_0}m^{\gamma}_{|u-m^{\gamma}_u(B(x,2^{-k}))|}(B(x,2^{-k-1}))\nonumber\\
&\leq\sum_{k\geq k_0}m^{\gamma/c_d}_{|u-m^{\gamma}_u(B(x,2^{-k}))|}(B(x,2^{-k}))\nonumber\\
& \leq 2\sum_{k\geq k_0}\inf_{c\in \mathbb{R}} m^{\gamma'}_{|u-c|}(B(x,2^{-k}))\nonumber\\&
\end{align*}
and then from Lemma \ref{phi integral average} we get
\begin{equation}\label{generalized Lebesgue pt_Triebel}
|u(x)-m^{\gamma}_u(B(x,2^{-k_0}))|\leq c\sum_{k\geq k_0}\phi(2^{-k})\left(\dashint_{B(x,2^{-k+1})}g^p_k d\mu\right)^{\frac{1}{p}},
\end{equation}
where $\gamma'=\min\{\gamma,\gamma/c_d\}$.
Using the definition of $\phi$-Haj\l asz gradient,
\begin{eqnarray*}
|u(y)|=|u(y)-u(z)|\leq \phi(d(y,z)) g_{k_0}(y).
\end{eqnarray*}
So, for a.e. $y, z$ with $d(y,z)<2^{-k_0},$ $|u(y)|\leq c\phi(2^{-k_0})g_{k_0}(y)$, since $\phi$ is almost increasing.
Use inequality \eqref{for p less than 1} when $q\leq 1$ and H\"older's inequality when $q\geq 1$ to obtain, for a.e. $y\in B(x,2^{-m}),$
\[
|u(y)|\leq C\phi(2^{-k_0})\big(\sum_{j\geq k_0} g_j(y)^q\big)^{1/q}
\]
 Hence, using the properties of $\gamma$-median we obtain
 \begin{align}\label{gamma median_Triebel} |m^{\gamma}_u(B(x,2^{-k_0}))|
&\leq m^\gamma_{c\phi(2^{-k_0})\|(g_j)\|_{l^q}}(B(x,2^{-k_0})) \nonumber\\
&=c \phi(2^{-k_0}) m^\gamma_{\|(g_j)\|_{l^q}}(B(x,2^{-k_0}))\nonumber\\
&\leq c\phi(2^{-k_0})\left(\gamma^{-1}\dashint_{B(x,2^{-k_0})}\|(g_j)\|_{l^q}^pd\mu\right)^{\frac{1}{p}}\nonumber\\
&\leq c\sum_{k\geq k_0}\phi(2^{-k})\left(\gamma^{-1}\dashint_{B(x,2^{-k})}\|(g_j)\|_{l^q}^p d\mu\right)^{\frac{1}{p}}.
\end{align}
Summing \eqref{generalized Lebesgue pt_Triebel} and \eqref{gamma median_Triebel}
\begin{align}
1 \leq u(x) & \leq c\sum_{k\geq k_0}\phi(2^{-k})\left(\dashint_{B(x,2^{-k+1})}\|(g_j)\|_{l^q}^p d\mu\right)^{\frac{1}{p}} \nonumber\\
& \leq c\left(\sum_{k\geq k_0}h_\omega(2^{-k+1})^{-\frac{1}{p}}\phi(2^{-k})\right)\sup_{k\geq k_0}h_\omega(2^{-k+1})^{\frac{1}{p}}\left(\dashint_{B(x,2^{-k+1})}\|(g_j)\|_{l^q}^p d\mu\right)^{\frac{1}{p}}\nonumber\\
&\leq c\left(\sum_{k\geq k_0}\omega(2^{-k})^{-1}\right)\sup_{k\geq k_0}h_\omega(2^{-k+1})^{\frac{1}{p}}\left(\dashint_{B(x,2^{-k+1})}\|(g_j)\|_{l^q}^pd\mu\right)^{\frac{1}{p}}\nonumber\\
&\leq c \sup_{k\geq k_0}h_\omega(2^{-k+1})^{\frac{1}{p}}\left(\dashint_{B(x,2^{-k+1})}\|(g_j)\|_{l^q}^pd\mu\right)^{\frac{1}{p}}.
\end{align}
Now for every $x\in E$, there is a ball $B(x,2^{-k_x+1})$ with $k_x\geq k_0$ such that 
$$\frac{\mu(B(x,2^{-k_x+1})}{h_\omega(2^{-k_x+1})}\leq c \int_{B(x,2^{-k_x+1})}\|(g_j)\|_{l^q}^p d\mu .$$ 
By $5B$-covering lemma, there exists a countable family of disjoint balls $B_j=B(x_j,2^{-k_{x_j}+1})$ of radii $r_j=2^{-k_{x_j}+1}$, such that the dilated balls $5B_j$ cover the set $E$.
Now, summing over $i$ and using the disjointness of the balls, doubling condition, increasing property of $\phi$ as well as the admissible property of $\omega,$ we obtain
\begin{align*}
\sum_{j}\frac{\mu(5B_j)}{h_\omega(5r_j)}& \leq c\sum_{j}\frac{\mu(B_j)}{h_\omega(r_j)}\nonumber\\
& \leq \Vert g_j \Vert^p_{L^p(X, l^q)}.
\end{align*}
and hence $H^{\phi}_{c2^{-m}}(E)\leq C\Vert g_j \Vert^p_{L^p(X, l^q)}.$ Then by taking $\delta\rightarrow 0$ in \eqref{from lemma_Triebel} and applying \eqref{claim_Triebel} we obtain the desired result. \qed

{\bf{Proof of Theorem \ref{main theorem_Sobolev}:}}
For simplicity, let us assume that $R=2^{-m}$ for some $m\in\mathbb{Z}.$ Without loss of generality, we can assume that $x_0\notin E.$ 
By Lemma \ref{locally lipschitz}, for any $\delta>0,$ there exists a $\phi$-Lipschitz function $v\in M^{\phi,p}(X)$ such that $v=1$ on a neighbourhood of $E$ and 
\begin{equation}\label{epsilon}
\|v\|^p_{M^{\phi,p}(X)}< \capacity_{M^\phi_{p,\infty}}(E)+\delta .
\end{equation}
Now, take a bounded $L$-Lipschitz function $\psi$ which satisfies 
\begin{equation}
\psi=0 \quad\text{on}\quad B(x_0,\frac{1}{4}\dist(x_0,E))\quad\text{and}\quad \psi=1 \quad\text{outside}\quad B(x_0,\frac{1}{2}\dist(x_0,E)). 
\end{equation}
Consider, $u=v\psi $.\\
So, $u=1$ on $E$ and we claim that there exists a $\phi$-Haj\l asz gradient of $u$ which takes $0$ in $B(x_0,\frac{1}{4}\dist(x_0,E))$ and satisfies 
\begin{equation}\label{gradient inequality}
\|g\|_{L^p(X)}\leq c\|v\|_{M^{\phi,p}(X)}.
\end{equation}
In fact, for any $\phi$-Haj\l asz gradient of $h$ of $v,$ the function
\begin{equation}\label{gradient u}
g(t)=
\begin{cases}
  \|\psi\|_{L^\infty(X)}h+[\|\psi\|_{L^\infty(X)}+1][\phi(L^{-1})]^{-1}|v|  & \text{if $t\in X\setminus B(x_0,\frac{1}{4}\dist(x_0,E))$},\\
  0& \text{if $t\in B(x_0,\frac{1}{4}\dist(x_0,E))$}
 \end{cases}
\end{equation}
is a $\phi$-Haj\l asz gradient of $u.$ This can be proved by defining $g_k$ as in \eqref{gk} with $h_k$ replaced by $h$ therein, choosing $i$ with $2^{i-1}\leq L<2^i$ and by taking $g=\sup_k g_k.$\\
 Let $x\in E,$ $y\in B(x,2^{-m})$ and $z\in B(x_0,\frac{1}{4}\dist(x_0,E)).$ So, $d(y,z)\leq 2^{-m+1}.$ Since $u$ is continuous, $x$ is a Lebesgue point of $u.$ Using properties of $\gamma$-median, doubling property of $\mu$ and Lemma \ref{phi integral average} in a similar fashion to the estimate \eqref{generalized Lebesgue pt_Triebel}, we get
\begin{align}\label{phi generalized Lebesgue pt_Sobolev}
|u(x)-m^{\gamma}_u(B(x,2^{-m+1}))|&\leq c\sum_{k\geq m-1}\phi(2^{-k})\left(\dashint_{B(x,2^{-k+1})}g^p d\mu\right)^{\frac{1}{p}}.
\end{align}
Using the definition of $\phi$-Haj\l asz gradient,
\begin{eqnarray*}
|u(y)|=|u(y)-u(z)|\leq \phi(d(y,z)) g(y)
\end{eqnarray*}
and hence similar to \eqref{gamma median_Triebel}, we have
\begin{align}\label{phi gamma median} |m^{\gamma}_u(B(x,2^{-m+1}))|
\leq c\sum_{k\geq m-1}\phi(2^{-k})\left(\dashint_{B(x,2^{-k+1})}g^p d\mu\right)^{\frac{1}{p}}.
\end{align}
Summing \eqref{phi generalized Lebesgue pt_Sobolev} and \eqref{phi gamma median}
\begin{align}
1 \leq u(x) & \leq c\sum_{k\geq m-1}\phi(2^{-k})\left(\dashint_{B(x,2^{-k+1})}g^pd\mu\right)^{\frac{1}{p}} \nonumber\\
& \leq c\left(\sum_{k\geq m-1}h_\omega(2^{-k+1})^{-\frac{1}{p}}\phi(2^{-k})\right)\sup_{k\geq m-1}h_\omega(2^{-k+1})^{\frac{1}{p}}\left(\dashint_{B(x,2^{-k+1})}g^pd\mu\right)^{\frac{1}{p}}\nonumber\\
&\leq c\left(\sum_{k\geq m-1}\omega(2^{-k})^{-1}\right)\sup_{k\geq m-1}h_\omega(2^{-k+1})^{\frac{1}{p}}\left(\dashint_{B(x,2^{-k+1})}g^pd\mu\right)^{\frac{1}{p}}\nonumber\\
&\leq c \sup_{k\geq m-1}h_\omega(2^{-k+1})^{\frac{1}{p}}\left(\dashint_{B(x,2^{-k+1})}g^pd\mu\right)^{\frac{1}{p}}.
\end{align}
Now for  every $x\in E$, there is a ball $B(x,2^{-k_x+1})$ with $k_x\geq m-1$ such that 
$$\frac{\mu(B(x,2^{-k_x+1})}{h_\omega(2^{-k_x+1})}\leq c \int_{B(x,2^{-k_x+1})}g^pd\mu .$$
By $5B$-covering lemma, pick a collection of disjoint balls $B_i=B(x_i,2^{-k_{x_i}+1})$ of radii $r_i=2^{-k_{x_i}+1}\leq 2^{-m+2}$ such that $E\subset \cup_i5B_i.$ Now, summing over $i$ and using the disjointness of the balls, doubling condition as well as the increasing property of $\phi,$ we obtain 
\begin{align*}
\|g\|^p_{L^p(X)}&\geq c \sum_i\frac{\mu(B_i)}{h_\omega(r_i)}\nonumber\\
& \geq c \sum_i\frac{\mu(5B_i)}{h_\omega(5r_i)}\nonumber
\end{align*}
and hence $H^{h_\omega}_{5\cdot 2^{-m+2}}\leq c\|g\|^p_{L^p(X)}.$ Taking $\delta\rightarrow 0$ in \eqref{epsilon} and applying \eqref{gradient inequality} we obtain desired result. \qed

\vspace{1cm}
\smallskip
{\small Nijjwal Karak}\\
\small{Department of Mathematics,}\\
\small{Birla Institute of Technology and Science-Pilani, Hyderabad Campus,}\\
\small{Hyderabad-500078, India} \\
{\tt nijjwal@gmail.com ; nijjwal@hyderabad.bits-pilani.ac.in}\\
\\
{\small Debarati Mondal}\\
\small{Department of Mathematics,}\\
\small{Birla Institute of Technology and Science-Pilani, Hyderabad Campus,}\\
\small{Hyderabad-500078, India} \\
{\tt p20200038@hyderabad.bits-pilani.ac.in}
\end{document}